\documentclass[11pt]{amsart}
\usepackage{palatino}
\usepackage{amsfonts}
\usepackage{amssymb}
\usepackage{amscd}
\usepackage[breaklinks,bookmarksopen,bookmarksnumbered]{hyperref}
\usepackage[dvips]{graphicx}
\input xy
\xyoption{all}
\usepackage[all]{xy}

\newtheorem{theorem}{Theorem}[section]
\newtheorem{proposition}[theorem]{Proposition}
\newtheorem{corollary}[theorem]{Corollary}
\newtheorem{lemma}[theorem]{Lemma}
\theoremstyle{definition}
\newtheorem{definition}[theorem]{Definition}
\newtheorem{remark}[theorem]{Remark}
\newtheorem{example}[theorem]{Example}

\newtheorem{conjecture}[theorem]{Conjecture}
\newtheorem{conjecture/question}[theorem]{Conjecture/Question}
\newtheorem{question}[theorem]{Question}
\newtheorem{remark/definition}[theorem]{Remark/Definition}
\newtheorem{terminology/notation}[theorem]{Terminology/Notation}

\def\ZZ{{\mathbb Z}}
\def\GG{{\textbf G}}

\def\PP{{\textbf P}}
\def\OO{\mathcal{O}}

\def\cB{\mathcal{B}}
\def\cA{\mathcal{A}}
\def\F{\mathcal{F}}

\def\E{\mathcal{E}}

\def\I{\mathcal{I}}

\def\cM{\mathcal{M}}

\def\cU{\mathcal{U}}

\def\Pic0{{\rm Pic}^0(X)}

\def\mm{\overline{\mathcal{M}}}

\newcommand {\ra}{\rightarrow}

\begin{document}

\title{The maximal rank conjecture and rank two Brill-Noether theory}

\author[G. Farkas]{Gavril Farkas}

\address{Humboldt-Universit\"at zu Berlin, Institut F\"ur Mathematik,  Unter den Linden 6
\hfill \newline\texttt{}
 \indent 10099 Berlin, Germany} \email{{\tt farkas@math.hu-berlin.de}}
\thanks{}

\author[A. Ortega]{Angela Ortega}
\address{Humboldt-Universit\"at zu Berlin, Institut F\"ur Mathematik,  Unter den Linden 6
\hfill \newline\texttt{}
 \indent 10099 Berlin, Germany} \email{{\tt ortega@math.hu-berlin.de}}

\maketitle
\begin{center}
\small{\emph{To the  memory  of  \ Eckart Viehweg}}
\end{center}
\section{Introduction}

The classical Brill-Noether theory of linear series on a curve $[C]\in \cM_g$, which describes the cycles
$W^r_d(C):=\{L\in \mbox{Pic}^d(C): h^0(C, L)\geq r+1\}$, is one of the celebrated successes in the theory of algebraic curves. There have been numerous attempts to extend this theory to vector bundles of higher rank, and the subject of this paper is the interplay between Koszul cohomology of line bundles and Brill-Noether phenomena for vector bundles of rank $2$ on curves. Let $\cU_C^s(2, d)$ be the moduli space of stable vector bundles on $C$ of rank $2$ and degree $d$. For each $k\geq 0$, we consider the \emph{Brill-Noether cycle}
$$\mathcal{BN}_C(d, k):=\{E\in \cU_C^s(2, d): h^0(C, E)\geq k\}.$$
It is well-known that $\mathcal{BN}_C(d, k)$ has the structure of a determinantal subscheme of $\cU_C^s(2, d)$, and accordingly, each of its irreducible components is of dimension at least equal to the $\emph{Brill-Noether number}$
$\beta_g(d, k):=4g-3-k\bigl(k-d+2g-2\bigr).$
The expectation that for a general curve $[C]\in \cM_g$, the variety $\mathcal{BN}_C(d, k)$ is non-empty precisely when $\beta_g(d, k)\geq 0$, is false, and there are few uniform statements concerning the geometry of $\mathcal{BN}_C(d, k)$.
 A remarkable exception to such erratic behaviour is the highly interesting case of rank $2$ vector bundles with canonical determinant, which is clarified in \cite{T3}.
\vskip 4pt

To a bundle $E\in \mathcal{SU}_C(2, L)$ with $\mbox{det}(E)=L\in \mbox{Pic}(C)$ and $h^0(C, E)=p+3\geq 4$, following a construction introduced in \cite{V3} and developed in \cite{AN}, one associates a non-trivial Koszul class $[\zeta(E)]\in K_{p, 1}(C, L)$. In this way, one establishes  a dictionary between rank $2$ \emph{Brill-Noether theory} and the \emph{Koszul geometry} of $C$. For $p=1$, this procedure specializes to a more classical construction \cite{BV}, \cite{M2}, \cite{GMN}, that assigns to a vector bundle $E\in \mathcal{SU}_C(2, L)$ with $h^0(C, E)=4$, a quadric $Q_E\in \mbox{Sym}^2 H^0(C, L)$ of rank at most $6$, containing the image $\phi_L(C)$ of $C$ under the map induced by $|L|$.
\vskip 3pt

The starting point of our investigation was an attempt to translate, via this dictionary, various syzygetic results for curves in the style of \cite{AF}, \cite{F3}, into dimensionality problems for $\mathcal{BN}_C(d, k)$. For $k\leq 3$ and a general $[C]\in \cM_g$, the Brill-Noether locus $\mathcal{BN}_C(d, k)$ is irreducible and of the expected dimension $\beta_g(d, k)$, see \cite{T1}.  The first case not governed by classical Brill-Noether theory is $k=4$, and we note that $$\beta_g(d, 4)=4d-4g-11.$$ It is natural to ask whether in this case too, the Brill-Noether number, determines the non-emptiness of $\mathcal{BN}_C(d, 4)$. Teixidor \cite{T2} has provided almost optimal answers to this question, and we summarize her results for a general curve $[C]\in \cM_g$:
$$ \mathcal{BN}_C(d, 4)\neq \emptyset, \ \ \mbox{ provided that  } d\geq \begin{cases}
2a+3, \text{ if } g=2a \Leftrightarrow \beta_g(d, 4)\geq 1;\\
2a+5, \text{ if } g=2a+1 \Leftrightarrow \beta_g(d, 4)\geq 5.\\
 \end{cases}
 $$
This leaves the case $g=2a+1$ and $d=2a+4$, as the only remaining possibility when $\beta_g(d, 4)\geq 0$. We prove the following result:
\begin{theorem}\label{existence}
For a general curve $[C]\in \cM_{2a+1}$, the Brill-Noether locus $\mathcal{BN}_C(2a+4, 4)$ is non-empty and has at least one component of dimension $2$.
\end{theorem}
Note that since $\beta_g(d, 4)=1$, unlike in the case $k\leq 3$, the Brill-Noether number no longer predicts the dimension of $\mathcal{BN}_C(d, 4)$. This is a phenomenon which propagates beyond control as $k$ grows, and appears for the first time when $k=4$.
This result combined with \cite{T2}, settles the existence problem for bundles of rank $2$ with $4$ sections:
\begin{corollary} For a general curve $[C]\in \cM_g$, we have that $\mathcal{BN}_C(d, 4)\neq \emptyset$ whenever $\beta_g(d, 4)\geq 0$.
\end{corollary}

Using the already mentioned connection between coherent systems $(E, V)$, where $E\in \cU_C(2, d)$ and $V\in G(4, H^0(C, E))$ on one side, and the non-vanishing of the cohomology group $K_{1, 1}(C, \mbox{det}(E))$ on the other, Theorem \ref{existence} is implied by the following:

\begin{theorem}\label{quad4}
For a general curve $[C]\in \cM_{2a+1}$, the locus of special linear series
$$\mathfrak{Koszul}(C):=\{L\in W^4_{2a+4}(C): \mathrm{Sym}^2 H^0(C, L)\stackrel{\nu_2(L)}\longrightarrow H^0(C, L^{\otimes 2})\ \ \mbox{ is not injective}\}$$
has at least one component of dimension $2$, whose general element corresponds to a complete base point free linear series, which cannot be written as $L=A_1\otimes A_2$, with $A_1, A_2\in W^1_{a+2}(C)$.
\end{theorem}
Assuming Theorem \ref{quad4}, the corresponding vector bundle $E\in \mathcal{BN}_C(2a+4, 4)$ is constructed as a twist of a \emph{Lazarsfeld} bundle on $C$. Precisely, for $L\in \mathfrak{Koszul}(C)$, we take $E:=M_{W}\otimes L$, where $W\in G(3, H^0(C, L))$ is a suitably chosen subspace such that $\mbox{Ker } \nu_2(L)\cap \bigl(W\otimes H^0(C, L)\bigr)\neq 0$. This method of constructing $E$ is the first instance of a general construction of vector bundles starting from non-trivial Koszul cohomology classes of small rank \cite{vB}, \cite{AN}. We refer to Section 5 for details.

\vskip 4pt

Next we turn to Mercat's generalization of Clifford's inequality. For a semistable vector
bundle $E$ of rank $2$ on $C$ and slope $\mu(E)$, Mercat \cite{Me} made an interesting prediction concerning its number of sections in terms of the Clifford index of the curve:
\begin{equation}\label{mercat1}
\mbox{If } \mathrm{Cliff}(C) + 2 \leq \mu(E) \leq g-1, \ \mbox{then }\
h^0(C, E) \leq 2+\mu(E) -\mathrm{Cliff}(C).
\end{equation}
\begin{equation}\label{mercat2}
\ \ \ \mbox{ If } \  1\leq \mu(E)\leq \mathrm{Cliff}(C)+2, \mbox{ then }\
h^0(C, E)\leq 2+\frac{1}{\mathrm{Cliff(C)+1}}\bigl(\mathrm{deg}(E)-2\bigr).
\end{equation}

The conjecture is inspired by the case when $E$ can be written as an extension
$$0\rightarrow A\rightarrow E\rightarrow A'\rightarrow 0,$$ where both line bundles $A, A'$ contribute to $\mathrm{Cliff}(C)$, in which case, (\ref{mercat1}) is an immediate consequence of Clifford's inequality applied to both $A$ and $A'$. For extensions of Clifford type inequalities to higher rank vector bundles and additional background, see \cite{LN}.
\vskip 3pt

We provide a counterexample to Mercat's Conjecture when $h^0(C, E)=4$, which was the simplest case when the answer was unknown:
\begin{theorem}\label{countermercat}
For each integer $a\geq 5$, there exist curves $[C]\in \cM_{2a+1}$ having maximal Clifford index $\mathrm{Cliff}(C)=a$, such that
$\mathcal{BN}_C(2a+3, 4)\neq \emptyset.$ In particular Mercat's Conjecture (\ref{mercat2}) fails for $C$.
\end{theorem}
The counterexamples to Mercat's Conjecture (also for $g=2a$, where $a\geq 6$, see Theorem \ref{K3surf2}), are sections of $K3$ surfaces lying in  certain Noether-Lefschetz loci. For the curves appearing in Theorem \ref{countermercat}, we observe that $\beta_g(2, d, 4)=-3$. The possibility that Mercat's Conjecture might fail for some curves of genus $11$ was already entertained in \cite{GMN} Remark 3.5 and \cite{LMN} Question 5.5. In fact, it was that particular suggestion in loc. cit. that drew our attention to this problem.

The proof of Theorem \ref{countermercat} uses again the observation that for a curve $C$ of genus $2a+1$ and gonality $a+2$, if $L\in W^4_{2a+3}(C)$ is a linear series such that the multiplication map
$\nu_2(L):\mbox{Sym}^2 H^0(C, L)\rightarrow  H^0(C, L^{\otimes 2})$ is not injective, then $\mathcal{BN}_C(2a+3, 4)\cap \mathcal{SU}_C^s(2, L)\neq  \emptyset$. More precisely, the locus of curves $[C]\in \cM_{2a+1}$ with $\mathcal{BN}_C(2a+3, 4)\neq \emptyset$ is set-theoretically equal to the Koszul locus
$$\mathfrak{Syz}_{g, 2a+3}^4:=\{[C]\in \cM_{2a+1}: \exists L\in W^4_{2a+3}(C)\ \mbox{ such that } K_{1, 1}(C, L)\neq \emptyset\}.$$
This is a virtual divisor in $\cM_{2a+1}$, which \emph{is not contained} in the Hurwitz divisor \cite{HM}
$$\cM_{2a+1, a+1}^1:=\{[C]\in \cM_{2a+1}: W^1_{a+1}(C)\neq \emptyset\}$$ of curves with a $\mathfrak g^1_{a+1}$. Curves $[C]\in \mathfrak{Syz}_{g, 2a+3}^4-\cM_{g, a+1}^1$ provide  counterexamples to (\ref{mercat2}).

\vskip 4pt
Even though there curves of maximal Clifford index not verifying (\ref{mercat2}), the question whether Mercat's inequalities (\ref{mercat1}) and (\ref{mercat2}) are true for a \emph{general} curve $[C]\in \cM_g$ remains a very stimulating one, and which can be naturally connected to the \emph{Maximal Rank Conjecture} (MRC) in the form that appears in \cite{AF}.

The original version of the MRC is due to Harris \cite{H} p. 79, and it amounts to the following: Let $C\subset \PP^r$ be a smooth curve of genus g and $\mbox{deg}(C)=d$, corresponding to a general point of the unique component of the Hilbert scheme ${\bf{Hilb}}_{d, g, r}$ mapping dominantly onto $\cM_g$ (that is, in the range $\rho(g, r, d)\geq 0$). Then the restriction maps
$$\nu_m(C):H^0(\PP^r, \OO_{\PP^r}(m))\rightarrow H^0(C, \OO_C(m))$$
have maximal rank. In particular the Hilbert function of $C$ is minimal. One can generalize Harris' Conjecture in two directions: Either $(a)$ by requiring that $[C]\in \cM_g$ be \emph{general in moduli} rather than in the Hilbert scheme, then conjecturing that the restriction maps to $C$ be of maximal rank with respect to \emph{all linear series} of type $\mathfrak g^r_d$, or $(b)$ by asking for the minimality not only of the Hilbert function but of the  entire \emph{graded Betti diagram} of $C$ (see Section 5 for how such a prediction can be correctly formulated). The generalization of Harris' Conjecture in direction $(a)$ was discussed in \cite{AF} and we briefly review it in Section 2. In particular, it predicts the following:
\vskip 4pt

\noindent \emph{Maximal Rank Conjecture $(MRC)_{g, d}^r$:\ }
We fix  integers $g, r, d\geq 1$ such that $$0\leq \rho(g, r, d)< 2d+2-g-{r+2\choose 2}.$$
For a general curve $[C]\in \cM_g$,
the map
$\nu_2(l):\mbox{Sym}^2(V)\rightarrow H^0(C, L^{\otimes 2})$
is injective for every linear series $l=(L, V)\in G^r_d(C)$.

\vskip 3pt
Returning to Theorem \ref{countermercat}, $(MRC)_{g, 2a+3}^4$ predicts that the syzygy locus $\mathfrak{Syz}_{g, 2a+3}^4$ is a proper subvariety
of $\cM_g$, and then it must be a divisor.

\begin{conjecture}\label{mrc4sectiuni}
Fix an integer $a\geq 5$ and a general curve $[C]\in \cM_{2a+1}$. Then $$K_{1, 1}(C, L)=0\ \mbox{ for every }\ L\in W^4_{2a+3}(C),$$ and the failure locus
$\mathfrak{Syz}_{g, 2a+3}^4$ is a divisor in $\cM_{2a+1}$. Consequently, Mercat's Conjecture (\ref{mercat2}) holds for all curves in the complement of $\mathfrak{Syz}_{g, 2a+3}^4$.
\end{conjecture}

Using Mukai's work \cite{M1}, we can confirm this expectation in one interesting case, namely that of curves of genus $11$, and answer Question 5.5 in \cite{LMN}:
\begin{proposition}\label{gen11}
The locus $\mathfrak{Syz}_{11, 13}^4:=\{[C]\in \cM_{11}: \exists L\in W^4_{13}(C) \ \mathrm{with}\ K_{1, 1}(C, L)\neq 0\}$ \ is an effective divisor in $\cM_{11}$.
In particular, $\mathcal{BN}_C(13, 4)=\emptyset$ for a general curve $[C]\in \cM_{11}$.
\end{proposition}
The above mentioned relation to syzygies, enables us to prove conjecture (\ref{mercat1}) for bounded genus:
\begin{theorem}\label{mercatcinci}
Mercat's Conjecture (\ref{mercat1}) holds for a general curve of genus $g\leq 16$.
\end{theorem}
The most beautiful case in the proof of Theorem \ref{mercatcinci} is when $[C]\in \cM_{15}$ and $h^0(C, E)=5$. In order to show that $\mathcal{BN}_C(19, 5)=\emptyset$, one must argue that if $$\phi_L:C\stackrel{|L|}\longrightarrow \PP^6$$ is one of the embeddings of $C$ by a linear series $L\in W^6_{19}(C)$ residual to a pencil of minimal degree, then $\phi_L(C)$ cannot lie on $5$ independent quadric hypersurfaces in $\PP^6$. Note that $4=\mbox{dim } \mbox{Sym}^2 H^0(C, L)-h^0(C, L^{\otimes 2})$ independent quadrics containing $\phi_L(C)$ come automatically, and we show that the existence of a \emph{fifth} quadric is a non-trivial condition in the moduli space $\cM_{15}$.

\vskip 4pt
To recapitulate, the original prediction  (\ref{mercat2}) is not true when formulated in terms of the original Clifford index, but both (\ref{mercat1}) and (\ref{mercat2}) are still expected to hold for general curves in moduli! It is customary to view the Koszul geometry of a curve as \emph{second order Brill-Noether theory}, in the sense that once all types of linear series $\mathfrak g^r_d$ on a curve have been prescribed, syzygies provide a finer analysis, distinguishing among curves with the same Brill-Noether behaviour. Our analysis lends some credence to the principle that this \emph{second order} BN analysis is connected in a precise forms (formulated in Section 5) to the \emph{rank two} BN theory of the curve and the various predictions on the two sides of this correspondence are remarkably compatible!
\vskip 3pt

As a word of caution however, proving $(MRC)_{g, d}^r$ when $\rho(g, r, d)\geq 1$ (let alone Conjecture \ref{minres}), seems considerably more difficult that proving the original Harris Conjecture. When $\rho(g, r, d)=0$ the two statements are equivalent, see \cite{F3} Theorem 1.5.
\vskip 3pt

We discuss the structure of the paper. In Section 2 we review the Maximal Rank Conjecture and some of its consequences. Section 3 contains the most important results of the paper. Using $K3$ surfaces, we disprove Mercat's Conjecture (\ref{mercat2})  (Theorems \ref{K3surf1} and \ref{K3surf2})  and set-up a link between rank $2$ vector bundles and MRC. We also prove Mercat's Conjecture (\ref{mercat1}) for general curves of bounded genus. In Section 4
we complete the proof of Theorem \ref{existence} concerning non-emptiness of Brill-Noether loci, while Section 5 is devoted entirely to Koszul cohomology and its applications to rank two Brill-Noether theory. We end the introduction by thanking Herbert Lange and Peter Newstead for pertinent comments made on an earlier version of this paper.

\section{The Maximal Rank Conjecture}

In \cite{AF}  a strong version of the \emph{Maximal Rank Conjecture} (MRC) for general curves has been formulated and its various applications to the birational geometry of $\mm_g$ have been presented. Since MRC will turn out to be also connected to rank two Brill-Noether theory, we begin by recalling, in a somewhat restricted form, the set-up from \cite{AF} Section 5.

We fix positive integers $g, r, d$ such that $\rho(g, r, d)\geq 0$, as well as a
general curve $[C]\in \cM_g$.  We may assume that $G^r_d(C)$ is smooth of dimension $\rho(g, r, d)$.
For a linear series
$l=(L, V)\in G^r_d(C)$ we denote by
$$
\nu_2(l): \mathrm{Sym}^2 (V) \rightarrow H^0(C, L^{\otimes 2})
$$
the multiplication map at the level of global sections. After choosing a Poincar\'e bundle on
$C\times \mbox{Pic}^d(C)$, following \cite{ACGH} Chapter VII, one
can construct vector bundles $\E_2$ and $\F_2$ over $G^r_d(C)$ with
$\mbox{rank}(\E_2)={r+2\choose 2}$ and $\mbox{rank}(\F_2)=h^0(C, L^{\otimes
2})=2d+1-g$, together with a bundle morphism $\nu_2:\E_2\rightarrow \F_2$, such
that for $l\in G^r_d(C)$ we have that
$$\E_2(l)=\mathrm{Sym}^2 (V) \ \mbox{ and } \F_2(l)=H^0(C, L^{\otimes 2}),$$ and
$\nu_2(l)$ is the multiplication map considered above. Since $[C]\in \cM_g$ satisfies Petri's theorem, $H^1(C, L^{\otimes 2})=0$, therefore by Grauert's theorem, $\F_2$ is locally free over $G^r_d(C)$.

\begin{conjecture}\label{strongmaxrank}
We fix integers $g, r, d\geq 1$ as above. For a general $[C]\in
\cM_g$, the locus
$$\mathfrak{Quad}_{g, d}^r(C):=\{l\in G^r_d(C): \nu_2(l) \mbox{ is not of maximal rank}\}$$
has the expected dimension as a determinantal variety, that is,
$$\mathrm{dim } \ \mathfrak{Quad}_{g, d}^r(C)=\rho(g, r,
d)-1-\bigl|2d+1-g-{r+2\choose 2}\bigr|,$$
where by convention, negative dimension means that $\mathfrak{Quad}_{g, d}^r(C)$ is empty.
\end{conjecture}
The most significant case of Conjecture \ref{strongmaxrank} is when we expect that $\mathfrak{Quad}_{g, d}^r(C)=\emptyset$, and we restate $(MRC)_{g, d}^r$ from the Introduction.
\begin{conjecture}\label{strongmaxrank2}
We fix integers $g, r, d\geq 1$ such that $$0\leq \rho(g, r, d)< 2d+2-g-{r+2\choose 2}.$$
For a general curve $[C]\in \cM_g$,
the  map $\nu_2(l)$ is injective for every  $l\in G^r_d(C)$.
\end{conjecture}
As discussed in \cite{AF}, various important cases of Conjecture \ref{strongmaxrank2} are known, see \cite{FP}, \cite{F3}, \cite{V1}. We feel that Conjecture \ref{strongmaxrank2} should be true, while the evidence for the stronger statement \ref{strongmaxrank} is perhaps less compelling and should be regarded more as an open question. It is reassuring to note that Conjecture \ref{strongmaxrank2} is compatible with classical Brill-Noether theory.
\begin{proposition}\label{rank4mrc}
$(MRC)_{g, d}^3$ holds. If $d\leq g+1$ and $[C]\in \cM_g$ is a Petri general curve, then $\nu_2(l)$ is injective for every $l\in G^3_d(C)$.
Thus $\mathfrak{Quad}_{g, d}^3(C)=\emptyset$.
\end{proposition}
\begin{proof} We fix $l:=(L, V)\in G^3_d(C)$ and use the elementary fact that if $\mbox{Ker } \nu_2(l)\neq 0$, then there exist pencils $A_1, A_2$ on $C$ such that $L=A_1\otimes A_2$. By Brill-Noether theory, $\mbox{deg}(A_i)\geq [(g+3)/2]$ for $i=1, 2$, hence $\mbox{deg}(L) \geq g+2$, which is a contradiction.
\end{proof}


\section{Mercat's conjecture}

We follow standard notation and denote by $\cU_C^s(n, d)$ (respectively $\cU_C(n, d)$) the moduli space of  stable (respectively semistable) vector bundles of rank $n$ and degree $d$ on $C$. If $L\in \mathrm{Pic}^d(C)$ is a line bundle,  we set $\mathcal{SU}_C(n, L):=\{E\in \mathcal{U}_C(n, d): \mbox{det}(E)=L\}$ and $\mathcal{SU}_C^s(n, L):=\mathcal{SU}_C(n,L)\cap \cU_C^s(n, d)$.

Recently, Lange and Newstead \cite{LN}
proposed a definition of the Clifford index of a higher rank vector bundle.  For $E\in \cU_C(n, d)$, the \emph{Clifford index} of $E$ is the quantity $$\gamma(E):=2+\mu(E)-\frac{2}{n}h^0(C, E)\geq 0.$$
By Serre duality, $\gamma(K_C\otimes E^{\vee})=\gamma(E)$. The \emph{higher Clifford indices} of $C$ are defined as
$$\gamma_n(C):=\mbox{min}\bigl\{\gamma(E): E\in \cU_C(n, d), \ \ \mu(E)\leq g-1, \ \ h^0(C, E)\geq 2n\bigr\}.$$
Note that $\gamma_1(C)=\mathrm{Cliff}(C)$ is the classical Clifford index of $C$. Several foundational properties of the invariants $\gamma_n(C)$ are studied in \cite{LN}. For instance the following inequality follows from the definition and is implicitly used in loc. cit.
\begin{lemma}
$\gamma_2(C)\leq \mathrm{Cliff}(C)$.
\end{lemma}
\begin{proof} We choose a line bundle $A$ on $C$ computing the Clifford index of $C$, that is, satisfying $\mbox{deg}(A)-2h^0(C, A)+2=\mathrm{Cliff}(C)$,
where $h^0(C, A)\geq 2$. We set $E:=A\oplus A$ and note that $\gamma(E)=\gamma(A)=\mathrm{Cliff}(C)$.
\end{proof}

An attempt to determine $\gamma_2(C)$ for a general curve $[C]\in \cM_g$, can be linked to an older conjecture of Mercat \cite{Me}. As already mentioned in the Introduction, for  a bundle $E\in \widetilde{\cU}_C(2, d)$ with $\mathrm{Cliff}(C) + 2 \leq \mu(E) \leq 2g-4 - \mathrm{Cliff}(C)$, it was predicted that
$$
h^0(C, E) \leq 2+\mu(E) - \mathrm{Cliff}(C).
$$
As pointed out in \cite{LN}, a consequence of (\ref{mercat1}) and (\ref{mercat2}) is the equality $\gamma_2(C)=\mathrm{Cliff}(C)$. A positive answer to Mercat's question, would show that, from the point of view of Clifford theory, special rank $2$ vector bundles are determined by special classical linear series.
Inequalities (\ref{mercat1}), (\ref{mercat2}) hold trivially when $h^0(C, E) \leq 3$, thus one may assume that $h^0(C, E) \geq 4$.
The following observation is essentially contained in \cite{Me}. We choose to make it explicit in order to make the bounds in  (\ref{mercat1}) and (\ref{mercat2}) transparent to ourselves:
\begin{lemma}
Let $E\in \cU_C(2, d)$ with $\mu(E)\leq g-1$.
If  $E$ contains a sub-pencil, then (\ref{mercat1}) holds.
\end{lemma}

\begin{proof}
Suppose that the vector bundle $E$ fits into an exact sequence
$$
0 \ra A \ra E \ra A' \ra 0,
$$
with $A$ a subbundle with $h^0(C, A)= 2$.  Then $h^1(C, A) = 2- \deg(A) + g-1 \geq 2$, if and only if $g-1 \geq \deg(A)$, but this last  inequality is satisfied by the semistability of $E$.  Since $4 \leq h^0(E) \leq h^0(A) + h^0(A')  $,
we get  $h^0(A') \geq 2$. \\
If $h^1(C, A')\geq 2$, then  both $A$ and $A'$ contribute to the Clifford index. It follows that
$$
h^0(C, E) \leq h^0(C, A) + h^0(C, A') \leq \frac{\deg(A)-\mathrm{Cliff}(C) +2}{2} + \frac{\deg(A')-\mathrm{Cliff}(C) +2}{2}$$
 $$=\mu(E) +2-\mathrm{Cliff}(C),
$$
that is,  inequality \eqref{mercat1} holds in that case.\\
Suppose $h^1(C, A') \leq 1$. Applying the definition of Clifford index to the bundle $A$, we obtain
$
\deg(A) \geq \mathrm{Cliff}(C) +2,
$
hence $h^1(C, A) = 2- \deg(A) +g-1 \leq g -\mathrm{Cliff}(C)-1 $.
On the other hand, by means of the long exact sequence in cohomology, we have
\begin{eqnarray*}
h^0(C, E) &=& h^1(C, E) +d - 2(g-1) \\
& \leq& h^1(C, A) + h^1(C, A') + d -2(g-1) \\
& \leq & d - \mathrm{Cliff}(C) -g +2  \\
& \leq& \frac{d}{2} - \mathrm{Cliff}(C) + 2
\end{eqnarray*}
where the last inequality follows by the hypothesis on $d$.
\end{proof}

From now on we shall assume that $E\in \cU_C(2, d)$ is globally generated and carries no sub-pencil.
We set $L:=\mbox{det}(E)\in \mathrm{Pic}^d(C)$ and consider the determinant map
$$
\lambda :  \bigwedge^2 H^0(C, E)  \ra H^0(C, L)
$$
The evaluation map $H^0(C, E) \otimes \OO_C \ra E$ induces a
morphism
$\phi_E:  x \mapsto E(x)\in G\bigl(2, H^0(C, E)^{\vee}\bigr) $.
Following  \cite{BV}, \cite{M2} we have a commutative diagram

\begin{equation} \label{diag}
\xymatrix@R=22pt@C=26pt{
C \ar[d]_{\phi_L} \ar[r]^{\phi_E \qquad} & G\bigl(2, H^0(C, E)^{\vee}\bigr) \ar@{^(->}[d] \\
\PP(H^0(C, L)^{\vee}) \ar[r]^{\PP(\lambda^{\vee})} &  \PP(\bigwedge^2 H^0(C, E)^{\vee})
}
\end{equation}
where  the vertical arrow on the right is the Pl\"ucker embedding and $\PP(\lambda^{\vee})$ is the map induced at the level of projective spaces
by the map dual to $\lambda$.
In order to estimate de number of sections of $L$ we will use the following lemma, which is
a direct consequence of \cite{PR} Lemma 3.9. We formulate it in a way that is compatible with (\ref{diag}).
\begin{lemma} \label{PRlemma}
Let $E$ be a globally generated rank 2 vector bundle on $C$ without sub-pencils. Then
$$\mathrm{dim}\bigl( \mathrm{Im}\ \lambda\bigr) \geq 2h^0(C, E)-3.$$
 In particular, $h^0(C, L)\geq 2h^0(C, E)-3$ \ and \ $\mathrm{dim }\ ( \mathrm{Im}\ \PP(\lambda^{\vee}))\geq 2h^0(C, E)-4$.
\end{lemma}
\begin{proof} We identify $G(2, H^0(C, E))\subset \PP\bigl(\bigwedge^2 H^0(C, E)\bigr)$ with the set of decomposable tensors $s\wedge t$, where $s, t\in H^0(C, E)$. The assumption that $E$ carries no sub-pencils implies that  \ $\PP(\mbox{Ker}\ \lambda) \cap G(2, H^0(C, E))=\emptyset$, and the claimed inequality follows.
\end{proof}

Inside the dual projective space $\PP\bigl(\bigwedge^2 H^0(C, E)\bigr)$, we identify $\PP\bigl(\mbox{Ker } \lambda\bigr)$ with the space of hyperplanes in $\PP\bigl(\bigwedge^2 H^0(C, E)^{\vee}\bigr)$ containing the span $\langle \phi_L(C)\rangle$. We set
$$\GG:=G\bigl(2, H^0(C, E)^{\vee}\bigr),\ \ \PP:=\PP\bigl(\bigwedge^2 H^0(C, E)^{\vee}\bigr)\  \mbox{ and }\ \Lambda:=\mbox{Im } \PP(\lambda^{\vee})\subset \PP.$$

Let us assume that  $h^0(C, E)=4$. Lemma \ref{PRlemma} implies that $\mbox{dim}(\mbox{Im } \PP(\lambda^{\vee})) \geq 4$, and
$$Q_E:=\PP\bigl(\lambda^{\vee}\bigr)^{-1} \bigl( G(2, H^0(C, E)^{\vee})\bigr)\in \mbox{Sym}^2 H^0(C, L)$$
is a quadric of rank at most $6$ containing $\phi_L(C)$. In particular, the multiplication map $$\nu_2(L):\mbox{Sym}^2 H^0(C, L)\rightarrow H^0(C, L^{\otimes 2})$$ is not injective. Equivalently $K_{1, 1}(C, L)=\mathrm{Ker}\ \nu_2(C, L)\neq 0$.

\noindent More generally, diagram (\ref{diag}) induces a pull-back morphism at the level of quadrics
$$\mathrm{res}_{C}: H^0\bigl(\PP, \I_{\GG/\PP}(2)\bigr)\rightarrow \mathrm{Ker } \ \nu_2(C, L).$$
To link the geometry of $E$ to a syzygy type statement, we estimate the rank of $\mbox{res}_C$.

\begin{proposition}\label{restrquadrics}
Assume $E$ is a globally generated rank $2$ vector bundle on $C$, without sub-pencils and with $h^0(C, E)\leq 5$. Then the map $\mathrm{res}_C$
 is injective.
\end{proposition}
\begin{proof} We begin with a \emph{Pl\"ucker quadric} $Q\in H^0(\PP, \I_{\GG/\PP}(2))$, that is, a rank $6$ quadric corresponding to a
$4$-dimensional quotient of $H^0(C, E)^{\vee}$.  The dual  $Q^{\vee}\subset \PP\bigl(\bigwedge^2 H^0(C, E)\bigr)$ is $4$-dimensional and contained in the dual Grassmannian $G\bigl(2, H^0(C, E)\bigr)$.
Since $E$ contains no sub-pencils, it follows that $\PP(\mbox{Ker } \lambda)\cap Q^{\vee}=\emptyset$, that is, no hyperplane $H$
$$\Lambda\subset H\subset \PP$$ is tangent to $Q$. But this clearly implies that $\mbox{res}_C(Q)\neq 0$, for otherwise it would imply that
$\Lambda \subset \mbox{Sing}(Q)$. This is impossible based on  dimension reasons. Since every quadric containing $G(2, 5)\subset \PP^9$ is a Pl\"ucker quadric
this finishes the proof.
\vskip 3pt

\end{proof}
\vskip 3pt

We discuss how Proposition \ref{restrquadrics} can be applied to study Mercat's Conjecture. When $h^0(C, E)=4$, inequality (\ref{mercat1}) is vacuous for curves of maximal Clifford index, while (\ref{mercat2}) breaks into two vanishing statements depending on the parity of $g$:
\begin{question}\label{mc2}
For $[C]\in \cM_{2a+1}$ with $\mathrm{Cliff}(C)=a$, is it true that $\mathcal{BN}_C(2a+3, 4)=\emptyset$?
 For a curve $[C]\in \cM_{2a}$ with $\mathrm{Cliff}(C)=a-1$, is it true that
$\mathcal{BN}_C(2a+1, 4)=\emptyset$?
\end{question}

The answer to both these questions is negative. Using the surjectivity of the period map for $K3$ surfaces in the style of \cite{F1}, \cite{K}, we construct curves of maximal gonality and prescribed degree and genus, lying on $K3$ surfaces in $\PP^4$.

\begin{theorem}\label{K3surf1}
For each integer $a\geq 5$, there exist smooth curves $C\subset \PP^4$ with $\mathrm{deg}(C)=2a+3$,  $g(C)=2a+1$ and maximal Clifford index
$\mathrm{Cliff}(C)=a$, such that $C$ lies on a smooth complete intersection $K3$ surface. As a consequence, $\mathcal{BN}_C(2a+3, 4)\neq \emptyset$ and Mercat's Conjecture fails for $C$.
\end{theorem}
\begin{proof}
We use \cite{K} Theorem 6.1 to construct a curve $C\subset S\subset \PP^4$, lying on a smooth complete intersection surface of type $(2, 3)$ such that $\mbox{Pic}(S)=\mathbb Z\cdot H\oplus \mathbb Z\cdot C$, where $H^2=6,\ H\cdot C=2a+3$ and $C^2=4a$. Since $h^1(C, \OO_C(1))\geq 2$, it follows that
$\OO_C(1)$ contributes to $\mathrm{Cliff}(C)$, hence
$\mathrm{Cliff}(C)\leq \mathrm{Cliff}(\OO_C(1))=2a-5$. We aim to show that $\mathrm{Cliff}(C)=a$, that is, $C$ has maximal possible Clifford index.

\vskip 3pt
Assume by contradiction that $\mathrm{Cliff}(C)<a$, which means that
$\mbox{Cliff}(C)$ is computed by a line bundle which comes from $S$. Note by direct calculation that $S$ carries no $(-2)$ curves, in particular $C$ has Clifford dimension $1$. We reason along the lines of \cite{F1} Theorem 3. Using \cite{GL} we infer that there exists a curve $D\subset S$, satisfying
\begin{eqnarray}\label{conditions}
  & & h^0(S, \OO_S(D))=h^0(C, \OO_C(D))\geq 2, \nonumber\\
 & & h^0(S, \OO_S(C-D))=h^0(C, K_C(-D))\geq 2,\\
 & & C\cdot D\leq g-1, \nonumber
\end{eqnarray} such that
$\mathrm{Cliff}(C)=\mathrm{Cliff}(D\otimes \OO_C)=D\cdot C-D^2-2$. In particular, such a divisor $D \equiv mH + nC$,
with $m,n \in \ZZ$
must verify the inequalities:
$$
\begin{array}{rl}
\mbox{(i)} & D\cdot H= 6m + dn >2 \\
\mbox{(ii)} &  md + (2n-1 )(g-1) \leq 0 \\
\mbox{(iii)} & 3m^2 +mnd + n^2(g-1) \geq 0
\end{array}
$$
We claim that there exist no divisors $D\subset S$ with $D^2 > 0$, satisfying (i)-(iii).

{\it Case $n<0$. } From (iii), we have that either $m < -n$ or $m >-\frac{2a}{3} n$. In the first case, by using inequality (i) we get
$$
2< -6n +d n =n(2a - 3),
$$
which is a contradiction since $n<0$ and $a\geq5$. Suppose $m > - 2an/3>0$. Inequality (ii) implies
that $n(2-d/3)<1$, that is,  $ (-n) (2a-3) < 3 $. Hence $2a-3 <3 $, which contradicts the hypothesis $a \geq 5$.

{\it Case $n>0$. }  Again, from condition (iii), we have either $m<-\frac{2a}{3}n$ or $m> -n$. In the first case, using (i)
we obtain $2 < n(d-4a)$, which is impossible since $d=2a+3<4a$. Suppose now that $-n<m<0$.
From (ii) we have that $2a(2n-1) \leq  -md <  nd$, which implies  $n < \frac{2a}{ 4a-d} = \frac{2a}{2a-3} < 2.$
Then $n=1 > -m >0$, therefore the case $n>0$ does not occur.

{\it Case $n=0$. } From (ii), $m \leq \frac{g-1}{d} = \frac{2a}{2a+3} <1$, but this yields a contradiction since, $m>0$.
This completes the proof of the claim.
\vskip 3pt

We are left with checking that $\mbox{Cliff}(\OO_C(D))\geq a$, for all primitive effective classes $D\in \mathrm{Pic}(S)$ such that $D^2=0$. By direct calculation, either $D\equiv C-D$, in which case $\mbox{Cliff}(\OO_C(D))=D\cdot C-D^2-2=2a-5\geq a$,
or else, $D\equiv 2a H-3C$, hence $D\cdot C>g-1$, and $D$ cannot compute $\mathrm{Cliff}(C)$.
\end{proof}

For genus $g=2a$, we have an analogous result in a similar range. We skip details:
\begin{theorem}\label{K3surf2}
For $a\geq 6$, there exist smooth curves $C\subset \PP^4$ with $\mathrm{deg}(C)=2a+1$, $g(C)=2a$ and maximal Clifford index $\mathrm{Cliff}(C)=a-1$, such that $C$ is contained in a smooth $(2, 3)$ complete intersection $K3$ surface. It follows that $\mathcal{BN}_C(2a+1, 4)\neq \emptyset$.
\end{theorem}
\vskip 3pt
It is important to realize that although (\ref{mercat2}) (and very probably prediction (\ref{mercat1}) as well), fail for certain Brill-Noether general curves, we still expect both  Mercat conjectures to be valid for the generic curve. Theorem \ref{K3surf1} should be interpreted as stating that the failure locus of (\ref{mercat2}) is \emph{not} a Brill-Noether locus in the classical sense, but rather a \emph{Koszul subvariety} on $\cM_g$ in the style of \cite{F3}, \cite{F5}. Precisely, the locus in $\cM_{2a+1}-\cM_{2a+1, a+1}^1$ where inequality (\ref{mercat2}) does not hold, can be described as
$$\mathfrak{Syz}_{g, 2a+3}^4:=\{[C]\in \cM_{2a+1}: \exists L\in W^4_{2a+3}(C)\ \mbox{ such that }\ K_{1, 1}(C, L)\neq 0\}.$$
This is a virtual Koszul divisor. Using the terminology of Section 2, a point $[C]$ lies in $\mathfrak{Syz}_{g, 2a+3}^4$ if and only if $\mathfrak{Quad}^4_{g, 2a+3}(C)\neq \emptyset$. Noting that $\rho(g, 4, 2a+3)=2a-9$, whereas $h^0(C, L^{\otimes 2})=2a+6$, one computes that the virtual dimension of $\mathfrak{Quad}_{g, 2a+3}^4(C)$ as a determinantal variety, is equal to $-1$. Since it is not difficult to provide examples of \emph{embedded} curves $C\subset \PP^4$ of genus $g(C)=2a+1$ and $\mbox{deg}(C)=2a+3$, which lie on a single quadric such that $\OO_C(1)\in \mathfrak{Quad}_{g, 2a+3}^4(C)$ is an isolated point, one infers that only two scenarios are possible:

\begin{enumerate}
\item $\mathfrak{Syz}_{g, 2a+3}^4$ is a divisor inside $\cM_g$, that is, $K_{1, 1}(C, L)=0$ for a general curve $[C]\in \cM_g$ and for \emph{every} $L\in W^4_{2a+3}(C)$, or
\item $\mathfrak{Syz}_{g, 2a+3}^4=\cM_g$.
\end{enumerate}
Conjecture $(MRC)_{2a+1, 2a+3}^4$ predicts that the second possibility does not appear. In any event,  the case of $\PP^4$ ought to be one of the more manageable situations for testing MRC in arbitrary genus.  We can confirm this expectation for $a=5$.
\vskip 3pt
\noindent
\emph{Proof of Theorem \ref{gen11}.} Assume by contradiction that for a general curve $[C]\in \cM_{11}$ there exists a linear series $L\in W^4_{13}(C)$ such that $C\stackrel{|L|}\hookrightarrow \PP^4$ lies on a quadric $Q\subset \PP^4$. We claim that $Q$ must be smooth, because otherwise, $\mbox{rank}(Q)\leq 4$, and then $L$ is expressible as the sum of two pencils. This contradicts the fact that $\mbox{gon}(C)=7$. After counting dimensions, we observe that there exists $X\in |I_{C/\PP^4}(3)|$, which does not contain $Q$, and such that $S:=Q\cap X\subset \PP^4$ is a smooth $K3$ surface. By direct calculation, we check that $h^0(S, \OO_S(H-C))\geq 2$ and
$(H-C)^2=0$, that is, $S$ is an elliptic $K3$ surface. This contradicts the main result of \cite{M1}, where it has been shown that a general curve of genus $11$ lies on a single $K3$ surface of degree $20$, which moreover is general in its moduli space, in particular it has Picard number one.
\hfill $\Box$

\vskip 5pt
We next turn to the case of globally generated vector bundles $E$ with  $h^0(C, E)=5$ having no sub-pencils. We set as usual $L:=\mbox{det}(E)$ and then $h^0(C, L)\geq 7$.
\vskip 2pt
\begin{remark}\label{5sections}
For a general curve $[C]\in \cM_{2a+1}$, Mercat's Conjecture holds for vector bundles with $5$ sections, if and only if $\mathcal{BN}_C(2a+5, 5)=\emptyset$.  Similarly, for even genus, Mercat's Conjecture for $h^0(C, E)=5$ holds in the case of a general curve $[C]\in \cM_{2a}$, if and only if $\mathcal{BN}_C(2a+3, 5)=\emptyset$.
\end{remark}

Via diagram (\ref{diag}), we use the existence of the  Pl\"ucker quadrics in the ideal of the curve $\phi_L(C)$ embedded by the determinant line bundle, to confirm (\ref{mercat1}) in bounded genus:
\vskip 4pt
\noindent
\emph{Proof of Theorem \ref{mercatcinci}:} We fix a general curve $[C]\in \cM_g$ and a globally generated rank $2$ vector bundle $E$ on $C$ with $\mathrm{Cliff}(C)+2\leq \mu(E)\leq g-1$
and $L:=\mbox{det}(E)\in \mbox{Pic}^d(C)$. Let us assume that inequality (\ref{mercat1}) does not hold, that is,
\begin{equation}\label{contradiction}
d<2\bigl(h^0(C, E)-2+\mathrm{Cliff}(C)\bigr).
\end{equation}
Then, as pointed out, $E$ admits no sub-pencils and $h^0(C, L)\geq 2h^0(C, E)-3$. Since $C$ satisfies the Brill-Noether theorem, one writes $\rho\bigl(g, 2h^0(C, E)-4, d\bigr)\geq 0$. Coupled with assumption (\ref{contradiction}), this forces $h^0(C, E)\leq 5$, and then, $h^0(C, E)=5, g=15$  and $d\leq 19$. There is no harm in assuming $d=19$, because if $\mathcal{BN}_C(19, 5)=\emptyset$, the same statement holds for lower degree by carrying out generic elementary transformations.

Therefore $E\in \mathcal{BN}_C(19, 5)$ and from Proposition \ref{restrquadrics}, one finds that
$$\mbox{dim } \mbox{Ker } \nu_2(L)\geq 5=\mbox{dim } H^0\bigl(\PP^9, \I_{G(2, 5)/\PP^9}(2)\bigr).$$  Using again that $C$ is Brill-Noether general, we observe that $h^0(C, L)=7$, \ $h^0(C, L^{\otimes 2})=\chi(C, L^{\otimes 2})=24$ and $A:=K_C\otimes L^{\vee}\in W^1_9(C)$ is a pencil of minimal degree. We infer that $\nu_2(L)$ is not surjective, and  there exists a vector bundle $F\in \mathcal{SU}^s_C(2, K_C)$ in an extension
$$0\rightarrow A\rightarrow F\rightarrow L\rightarrow 0,$$
such that $h^0(C, F)=h^0(C, A)+h^0(C, L)=2+7=9$. The proof that $F$ is stable is standard, cf. \cite{L} Prop. V.4. Applying \cite{T3}, one can assume that the \emph{Mukai-Petri map}
$$\mbox{Sym}^2 H^0(C, F)\rightarrow H^0(C, \mbox{Sym}^2 F)$$
is injective, which is absurd since $3g-3<h^0(C, F)\ \bigl(h^0(C, F)+1\bigr)/2$.
\hfill $\Box$
\vskip3pt

In the same spirit, we can link inequality (\ref{mercat1}) to a MRC statement.

\begin{proposition} \label{det7}
Let $[C]\in \cM_{g}$ be general. Mercat's Conjecture (\ref{mercat1}) for vector bundles $E$ with $h^0(C, E)=5$ is a consequence of the Maximal Rank Conjecture.
\end{proposition}
\begin{proof}
We sketch only the odd genus case, and write $g=2a+1$. From Remark \ref{5sections} we know that it is enough to show that $\mathcal{BN}_C(2a+5, 5)=\emptyset$. If $E\in \cU_C(2, 2a+5)$ satisfies $h^0(C, E)=5$, then we know from Proposition \ref{restrquadrics} that the image $\phi_L(C)$ induced by the determinant line bundle,  lies on at least $5$ quadrics coming from the equations of the Grassmannian $G(2, 5)\subset \PP^9$. We set $r:=h^0(C, L)-1\geq 6$.  Over the variety $G^r_{2a+5}(C)$ of linear series $\mathfrak g^r_{2a+5}$ there exists a morphism of vector bundles $\nu_2: \E_2\rightarrow \F_2$ which globalizes the multiplication maps $\nu_2(l)$, for $l=(L, V)\in G^r_{2a+5}(C)$. The Maximal Rank Conjecture predicts that the determinantal locus $$X_5(\nu_2):=\{l\in G^r_{2a+5}(C): \mbox{dim } \mbox{Ker }\nu_2(l)\geq 5\},$$ has expected dimension, that is, $X_5(\nu_2)=\emptyset$, hence no vector bundle $E$ with $h^0(C, E)=5$ can exist.
\end{proof}

To close, we record the form conjecture (\ref{mercat1}) takes for bundles with $6$ sections. Computing the appropriate degrees, one must show that $\mathcal{BN}_C(2a+7, 6)=\emptyset$ for a general curve $[C]\in \cM_{2a+1}$ and $\mathcal{BN}_C(2a+5, 6)=\emptyset$ for a  general curve
$[C]\in \cM_{2a}$.

\section{Existence of stable vector bundles of rank $2$ with $4$ sections}

We begin by describing all possible bundles $E\in \cU_C(2, 2a+4)$ on a Petri general curve $[C]\in \cM_{2a+1}$ having $h^0(C, E)=4$. There are two cases to distinguish. Assume first that $E$ is stable and globally generated. Then $E$ carries no sub-pencil and $L:=\mbox{det}(E)\in W^4_{2a+4}(C)$, cf. Lemma \ref{PRlemma} (see also \cite{GMN}). Using diagram (\ref{diag}), as before we obtain a quadric of rank at most $6$
\begin{equation}\label{rank6quadric}
Q_E\in \PP\ \mbox{Ker}\bigl\{  \nu_2(L):\mbox{Sym}^2 H^0(C, L)\rightarrow H^0(C, L^{\otimes 2})\bigr\}
\end{equation}
containing the image of $\phi_L(C)$ of the curve under the linear series $|L|$.

\vskip 3pt
Assume now that $E$ carries a sub-pencil. Since $\mbox{gon}(C)=a+2$, then necessarily, $E$ sits in an extension
\begin{equation}\label{pencilext}
0\rightarrow A\rightarrow E\rightarrow A'\rightarrow 0,
\end{equation}
where $A, A'\in W^1_{a+2}(C)$, and $h^0(C, E)=h^0(C, A)+h^0(C,  A')$.
In particular,  $E$ is strictly semistable, $h^0(C, E)=4$ and the multiplication map
$$\mu_0(A', K_C\otimes A^{\vee}): H^0(C, A')\otimes H^0(C, K_C\otimes A^{\vee})\rightarrow H^0(C, K_C\otimes A'\otimes A^{\vee}),$$
obtained by dualizing the boundary morphism $$\mbox{Ext}^1(A', A)\rightarrow \mbox{Hom}\bigl(H^0(C, A'), H^1(C, A)\bigr)$$
is not surjective (One notes that if $A\neq A'$, then by Riemann-Roch $h^0(C, K_C\otimes A^{\vee})=a$ and $h^0(C, K_C\otimes A'\otimes A^{\vee})=2a$, that is, $\mu_0(A', K_C\otimes A^{\vee})$ is a morphism between vector spaces of the same rank $2a$).

For a general curve $[C]\in \cM_{2a+1}$, the \emph{Brill-Noether curve} $W_{a+2}^1(C)$ is smooth, connected and of genus $$g':=1+\frac{a}{a+1}{2a+2\choose a}.$$ The associated map $\phi:\mm_{2a+1}\dashrightarrow \mm_{g'}$ given by $\phi([C]):=[W^1_{a+2}(C)]$, has been studied in some detail in \cite{F4}.  Intriguing questions, like that of describing geometrically the image of $\phi$ in $\mm_{g'}$, or of studying the (possibly empty) non-injectivity locus of $\phi$, remain however. In particular, it would be interesting to understand the geometric properties (e.g. Brill-Noether theory, automorphisms if any) of the curve $W^1_{a+2}(C)$. The previous condition, shows that $W_{a+2}^1(C)$ comes equipped with an interesting correspondence:

\begin{theorem}\label{correspondence}
Fix $a\geq 2$ and a general curve $[C]\in \cM_{2a+1}$. The locus of pairs of pencils
$$\mathfrak{S}_C:=\{(A, A')\in W^1_{a+2}(C) \times W^1_{a+2}(C): \mu_0(A', K_C\otimes A^{\vee})\ \mbox{ is not injective}\}$$
is a non-empty, symmetric correspondence on $W^1_{a+2}(C)$, disjoint from the diagonal.
\end{theorem}

From the Base Point Free Pencil Trick it follows that $(A, A')\in \mathfrak S_C$ if and only if $H^0(C, K_C-A-A')\neq 0$, which proves that $\mathfrak S_C$ is symmetric. Furthermore, since the multiplication maps $\mu_{0}(A, K_C\otimes A^{\vee})$ are injective for all
$A\in W^1_{a+2}(C)$, it follows that $\mathfrak S_C\cap \Delta_{W^1_{a+2}(C)}=\emptyset$. The non-trivial part of Theorem \ref{correspondence} is to show that $\mathfrak S_C\neq \emptyset$, and we shall prove this by degeneration. In order to carry this out, we need some preparation and recall a few basic facts about degenerations of multiplication maps on curves.

\vskip 4pt
We fix a pointed curve $[C, p]\in \cM_{g, 1}$.
If $l=(L, V)\in G^r_d(C)$ is a linear series, then the \emph{vanishing sequence} $\{a_i^l(p)\}_{i=0, \ldots, r}$ of $l$ at $p$ is obtained by ordering the positive integers $\{\mbox{ord}_p(\sigma)\}_{\sigma\in V}$.
If $L$ and $M$ are line bundles on
$C$, we denote by $$\mu_0(L, M): H^0(C, L)\otimes H^0(C, M)\rightarrow
H^0(C, L\otimes M)$$ the usual multiplication map. For any element $\rho \in H^0(C, L)\otimes H^0(C, M)$,
we write that $\mbox{ord}_p(\rho)\geq k$, if
$\rho$ lies in the span of elements of the form $\sigma\otimes
\tau$, where $\sigma \in H^0(C, L)$ and $\tau \in H^0(C, M)$ are such that
$\mbox{ord}_p(\sigma)+\mbox{ord}_p(\tau)\geq k$.  Suppose
$\{\sigma_i\}\subset H^0(L)$ and $\{\tau_j\}\subset H^0(M)$ are
bases of global sections with the property that
$\mbox{ord}_p(\sigma_i)=a_i^L(p)$ and
$\mbox{ord}_p(\tau_j)=a_j^M(p)$ for all $i$ and $j$. Then if $\rho
\in \mbox{Ker}\ \mu_0(L,M)$, there exist two pairs of integers
$(i_1,j_1)\neq (i_2,j_2)$ such that
$$\mbox{ord}_p(\rho)=\mbox{ord}_p(\sigma_{i_1})+\mbox{ord}_p(\tau_{j_1})=\mbox{ord}_p(\sigma_{i_2})+
\mbox{ord}_p(\tau_{j_2}).$$

Let $[C_0:=D_0\cup_{p_0} E_0]\in \Delta_1\subset \mm_{2a+1}$ be a stable curve, where $[D_0, p_0]\in \cM_{2a, 1}$ and $[E_0, p_0]\in \cM_{1, 1}$ are general pointed curves. Let $\textbf{M}$ denote the versal deformation space of $C_0$, thus $\textbf{M}\rightarrow \mm_{2a+1}$ can be regarded as an \'etale neighbourhood of $[C_0]\in \mm_{2a+1}$. We then consider the proper Deligne-Mumford stack $\sigma:\mathfrak{G}^1_{a+2}\rightarrow \textbf{M}$ of limit linear series $\mathfrak G^1_{a+2}$, as well as the induced projection $\sigma':\mathfrak G^1_{a+2}\times_{\textbf{M}} \mathfrak G^1_{a+2}\rightarrow \textbf{M}$.
\vskip 4pt

The key technical tool in the proof of Theorem \ref{correspondence} is the construction of a stack $$\nu:\mathfrak S\rightarrow \mathfrak{G}^1_{a+2}\times_{\textbf{M}} \mathfrak{G}^1_{a+2}$$ such that, loosely speaking, the fibres of $\mu:=\sigma'\circ \nu$ are the (degenerations of the) correspondences $\mathfrak S_C$, when $[C]\in \textbf{M}$. The construction of $\mathfrak S$ goes along the lines of  \cite{F2} Theorem 4.3, for which reason we shall be rather succint.

\begin{definition} The stack $\mu: \mathfrak S\rightarrow \textbf{M}$ has the following structure:
\vskip2pt

\noindent
$\bullet$ For $[C]\in \textbf{M}$ corresponding to a smooth curve, the points in the fibre $\mu^{-1}[C]$ are triples $(A, A', \rho)$, where $A, A'\in W^1_{a+2}(C)$ and $\rho\in \PP\ \mbox{Ker } \mu_0(A', K_C\otimes A^{\vee})$.
\vskip 3pt

\noindent
$\bullet$ For $[C]\in \textbf{M}$ corresponding to a singular curve $C:=D\cup_ p E$, where $[D, p]\in \cM_{2a, 1}$ and $[E, p]\in \cM_{1, 1}$, the fibre $\mu^{-1}[C]$ classifies elements $$\bigl(l, m, \ \rho_1,\  \rho_2\bigr),$$ where $m=\bigl\{(L'_D, V'_D), (L'_{E},
V'_{E})\bigr\}\in \sigma^{-1}[C]$ is a limit $\mathfrak g^1_{a+2}$ on $C$, whereas
$$l=\bigr\{\bigl(K_{D}(2p)\otimes L_{D}^{\vee},
W_{D}\bigr), \ \bigl(\OO_{E}(4a\cdot p)\otimes
L_{E}^{\vee}, W_{E}\bigr)\bigr\}$$  is a limit $\mathfrak g_{3a-2}^{a-1}$ on $C$, which is complementary to a limit $\mathfrak g^1_{a+2}$ on $C$ having as aspects the line bundles
$L_D\in \mbox{Pic}^{a+2}(C)$ and $L_E\in \mbox{Pic}^{a+2}(E)$.
\vskip 3pt

Furthermore, we have elements
$$\rho_1\in \PP\mbox{Ker}\{V'_{D}\otimes W_{D}\rightarrow
H^0\bigl(D, K_{D}(2p)\otimes L'_D\otimes L_D^{\vee}\bigr)\},$$
$$ \rho_{2}\in \PP
\mbox{Ker}\{V'_{E}\otimes W_{E}\rightarrow H^0\bigl(E, \OO_E(4a\cdot p)\otimes L'_E\otimes L_E^{\vee} \bigr)\}$$ satisfying the compatibility relation
$\mbox{ord}_p(\rho_1)+\mbox{ord}_p(\rho_2)\geq 4a$.
\end{definition}
The morphism $\mathfrak S\stackrel{\mu}\rightarrow \textbf{M}$ factors through  $\sigma':\mathfrak G^1_{a+2}\times_{\textbf{M}} \mathfrak G^1_{a+2}\rightarrow \textbf{M}$ by forgetting the elements $\rho_1$ and $\rho_2$. Moreover, $\mathfrak S$ has a determinantal structure over $\textbf{M}$ and each fibre $\mu^{-1}([C])$ has dimension at least $1$. We are in a position to prove Theorem \ref{correspondence}:

\vskip 4pt
\noindent
\emph{Proof of Theorem \ref{correspondence}.} Keeping the notation above, it suffices to show that for $C:=D\cup_p E$, the fibre $\mu^{-1}([C])$
has at least one irreducible component of dimension $1$. This implies that $\mu(\mathfrak S)$ maps dominantly onto $\textbf{M}$. Since for a smooth curve $[C']\in \textbf{M}$, the fibre $\mu^{-1}([C'])$ is isomorphic to $\mathfrak S_{C'}$, the conclusion follows.

We choose $[D, p]\in \cM_{2a, 1}$ sufficiently general such that (i) $D$ satisfies Petri's Theorem, in particular, $W^1_{a+1}(D)$ is finite and reduced, (ii) $h^0(D, A\otimes A')=4$ for all  pencils $A\neq A'$ on $C$ of degree $a+1$ (cf. \cite{V1} 3.1), and (iii) $p\notin \mbox{supp}(A)$, for any $A\in W^1_{a+1}(D)$. We construct piece by piece an element $(l, m, \rho_1, \rho_2)\in \mu^{-1}[C]$ as follows: We set
$$m:=\bigl\{\bigl(A'(p), \ |V'_D|=p+|A'|\bigr),\ \ \bigl(A'_E(a\cdot p), \ |V'_E|=a\cdot p+|A'_E|\bigr)\bigr\},$$
where $A'\in W^1_{a+1}(D)$ and $A'_E\in \mbox{Pic}^2(E)$ are chosen arbitrarily. Then we take
$$l:=\bigl\{\bigl(L_D:=K_D(p)\otimes A^{\vee}, |L_D|\bigr),\ \ \bigl(\OO_E(3a\cdot p)\otimes A_E^{\vee}, (2a-2)\cdot p+|\OO_E((a+2)\cdot p)\otimes A_E^{\vee})|\bigr)\bigr\},$$
where $A\in W^1_{a+1}(C)-\{A'\}$, and $A_E\in \mbox{Pic}^2(E)$ is again arbitrary. Thus $l$ is a refined limit $\mathfrak g^{a-1}_{3a-2}$ on $C$ having vanishing sequence with respect to $C$ equal to $a^{l_D}(p)=(1, 2, \ldots, a)$. By varying $A, A'\in W^1_{a+1}(D)$ and $A_E, A'_E\in \mbox{Pic}^2(E)$, we fill-up an entire component of the fibre $(\sigma')^{-1}[C]$.

We now describe all possibilities of choosing $\rho_1, \rho_2$ compatible with $l$ and $m$. First,
$$\rho_1 \in \PP\mbox{Ker}\bigl\{H^0(D, A'(p))\otimes H^0\bigl(D, K_D(p)\otimes A^{\vee}\bigr)\rightarrow H^0\bigl(D, K_D(2p)\otimes A'\otimes A^{\vee}\bigl)\bigr\}$$ is uniquely determined corresponding to the non-zero section from $H^0(D, K_D-A-A')$. Clearly $\mbox{ord}_p(\rho_1)=3$, hence by compatibility $\mbox{ord}_p(\rho_2)\geq 4a-3$. After subtracting the base point $p\in E$, we find that $\rho_2$ must correspond to the unique non-zero element in the kernel of the multiplication map $$\mu_0(A'_E, \OO_E(4p)\otimes A_E): H^0(E, A'_E)\otimes H^0(E, \OO_E(4p)\otimes A_E^{\vee})\rightarrow H^0\bigl(E, \OO_E(4p)\otimes A_E'\otimes A_E^{\vee}\bigr).$$
This implies that $A_E\otimes A'_E=\OO_E(4p)$, hence $A_E\in \mbox{Pic}^2(E)$ can be freely chosen, and then $A'_E$ and $\rho_2$ are uniquely determined. All in all, $\mu^{-1}([C])$ has a $1$-dimensional component, which completes the proof. $\hfill$ $\Box$

\begin{theorem}
For $a\geq 4$ and a  general curve $[C]\in \cM_{2a+1}$, the determinantal variety
$$\mathfrak{Koszul}(C):=\{L\in W^4_{2a+4}(C): \nu_2(L)\ \mbox{ is not injective}\}$$
is non-empty and has a component of dimension $2$, corresponding to complete linear series $L\in W^4_{2a+4}(C)$ which cannot be written as  sums $L=A_1+A_2$, where $A_1, A_2\in W^1_{a+2}(C)$.
\end{theorem}
\begin{proof} Over the smooth $(2a-4)$-dimensional variety $G^4_{2a+4}(C)$ of linear series $\mathfrak g^4_{2a+4}$ on $C$, we construct vector bundles $\cA$ and $\cB$ having fibres $$\cA(L, V):=\mbox{Sym}^2(V)\ \mbox{  and } \cB(L, V):=H^0(C, L^{\otimes 2})$$ over each point
$(L, V)\in G^4_{2a+4}(C)$, where $L\in W^4_{2a+4}(C)$ and $V\subset H^0(C, L)$ is the corresponding $5$-dimensional space of sections. Clearly $\mbox{rank}(\cA)=15$ and $\mbox{rank}(\cB)=2a+8$. There exists a morphism of vector bundles $\nu_2:\cA\rightarrow \cB$, such that $$\nu_2(L, V):\mbox{Sym}^2(V)\rightarrow H^0(C, L^{\otimes 2})$$ is the multiplication map of sections. Every irreducible component of the degeneracy locus $\mathfrak{Quad}(\nu_2):=\{(L, V)\in G^4_{2a+4}(C): \nu_2(L, V) \mbox{ is not injective}\}$ has dimension at least $2=\mbox{dim } G^4_{2a+4}(C)-(2a+8-14)$.
\vskip 3pt

To show that $\mathfrak{Quad}(\nu_2)\neq \emptyset$, we use that the correspondence $\mathfrak S_C$ is non-empty, and choose a pair $(A, A')\in \mathfrak S_C$, such that $h^0(C, A\otimes A')=5$. The pencils $A$ and $A'$ are complete and base point free, and we pick $\{\sigma_0, \sigma_1\}\subset H^0(C, A)$
(respectively $\{\sigma_0', \sigma_1'\}\subset H^0(C, A')$) bases for the respective spaces of sections. Then the element
$$(\sigma_0\cdot \sigma_1')\cdot (\sigma_1\cdot \sigma_0')-(\sigma_0\cdot \sigma_0')\cdot (\sigma_1\cdot \sigma_1')\in \mathrm{Sym}^2 H^0(C, A\otimes A')$$
lies obviously in $\mbox{Ker }\nu_2(A\otimes A')$, that is, $A\otimes A'\in X(\nu_2)$. Let $Z\subset X(\nu_2)$ be an irreducible component such that $A\otimes A'\in Z$. Since $\mbox{dim}(Z)\geq 2$ and $\mathfrak S_C\subsetneqq W^1_{a+2}(C)\times W^1_{a+2}(C)$, necessarily, the general point of $Z$ corresponds to a complete linear series $L\in W^4_{2a+4}(C)$, which cannot be expressed as a sum of two pencils.
\end{proof}

To each $L\in \mathfrak{Koszul}(C)$ as above, with an element $0\neq q_L\in \mbox{Ker }\nu_2(L)$, we assign a vector bundle $E\in \mathcal{SU}_C(2, L)$ as follows, see also \cite{GMN}, \cite{vB}. Since
$\mbox{rank}(q_L)\leq 5$, there exists a subspace $W\in G(3, H^0(C, L))$ such that $$q_L\in \mbox{Sym}^2 H^0(C, L)\cap \bigl(W\otimes H^0(C, L)\bigr).$$ We define $E$ to be the kernel of the following evaluation map:
$$0\rightarrow E\rightarrow W\otimes L\rightarrow L^{\otimes 2}\rightarrow 0.$$
Clearly, $\mbox{det}(E)=L$ and $H^0(C, E)\supset \wedge^2 W\oplus \mathbb C \cdot q_L$, thus $h^0(C, E)\geq 4$. Moreover $E$ is globally generated.
\vskip 3pt

 The proof that $E$ is stable follows closely \cite{GMN} Theorem 3.2: An arbitrary quotient line bundle $A'$ of $E$ has $h^0(C, A')\geq 2$. Either $\mbox{deg}(A')>a+2$, which implies that $E$ is stable, or else, $\mbox{deg}(A')=a+2$ and $h^0(C, A')=2$. In the latter case, $E$ sits in an extension of type (\ref{pencilext}), in particular $L$ is expressible as a sum of two elements from $W^1_{a+2}(C)$, a contradiction. Therefore $E\in \mathcal{BN}_C(2a+4, 4)$.

\section{Applications of Koszul cohomology to rank $2$ vector bundles}

There is an interesting connection between vector bundles $E\in \cU_C(2, d)$ and syzygies of low rank in
the Koszul cohomology group $K_{h^0(E)-3, 1}\bigl(C, \mbox{det}(E)\bigr)$. The first instance of this equivalence, when $h^0(C, E)=4$, is classical and has been used in \cite{BV}, \cite{M2}, \cite{GMN}, as well as in this paper. We review a general construction which can be traced back to Voisin \cite{V3}, and has been explicitly worked out in \cite{AN}.

For a curve $C$ and a globally generated line bundle $L$ on $C$, the Koszul cohomology group $K_{p, 1}(C, L)$ can be defined as the cohomology of the complex:
$$\bigwedge^{p+1} H^0(C, L)\stackrel{d_{p+1, 0}}\longrightarrow \bigwedge^p H^0(C, L)\otimes H^0(C, L) \stackrel{d_{p, 1}}\longrightarrow \bigwedge^{p-1}H^0(C, L)\otimes H^0(C, L^{\otimes 2}).$$
If $M_L$ is the \emph{Lazarsfeld vector bundle} defined as the kernel of the evaluation map
$$ 0\rightarrow M_L\rightarrow H^0(C, L)\otimes \OO_C\stackrel{\mathrm{ev}}\rightarrow L\rightarrow
0,$$ a simple argument using the exact sequences
$$ 0\longrightarrow \bigwedge^a M_L\otimes L^{\otimes b} \rightarrow
\bigwedge^a H^0(C, L)\otimes L^{\otimes b}\longrightarrow \bigwedge^{a-1}
M_L\otimes L^{\otimes (b+1)}\longrightarrow 0$$ for various $a$ and
$b$, leads to an identification \cite{PR} p.506,
\begin{equation}\label{koszul}
 K_{p, 1}(C, L)=\frac{H^0(C, \wedge^p
M_L\otimes L^)}{\wedge^{p+1} H^0(C, L)
}.
\end{equation}

\begin{definition}
We say that a Koszul class $[\zeta]\in K_{p, 1}(C, L)$ has rank $\leq n$, if there exists a subspace $W\subset H^0(C, L)$ with $\mbox{dim}(W)=n$ and a representative
$\zeta \in \wedge^p W\otimes H^0(C, L)$.
\end{definition}

Let $E$ be a rank $2$ bundle on $C$ with $h^0(C, E)=p+3\geq 4$ and set $L:=\mbox{det}(E)$. We assume that the determinant map
$\lambda:\wedge^2 H^0(C, E)\rightarrow H^0(C, L)$ does not vanish on decomposable tensors, or equivalently, $E$ carries no sub-pencils.
Choosing a basis $(e_1, \ldots, e_{p+3})$ of $H^0(C, E)$, we introduce the subspace
$$W:=\bigl\langle s_2:=\lambda(e_1\wedge e_2), \ldots, s_{p+3}:=\lambda(e_1\wedge e_{p+3})\bigr\rangle
\subset H^0(C, L).$$ By assumption, $\mbox{dim}(W)=p+2$. Following \cite{AN} (2.1) and \cite{V3} formula (2.22), we define the syzygy
$$
\zeta(E):=\sum_{i<j} (-1)^{i+j}\  s_2\wedge \ldots \wedge \hat{s_i}\wedge \ldots \wedge \hat{s_j} \wedge \ldots \wedge s_{p+3}\otimes \lambda(e_i\wedge e_j) \in \wedge^p W \otimes H^0(C, L).
$$
It is shown in \cite{V3} Lemma 5, that $d_{p, 1}(\zeta(E))=0$, hence $[\zeta(E)]\in K_{p, 1}(C, L)$ gives rise to a non-trivial Koszul class of rank $p+2$.
\begin{remark} When $h^0(C, E)=4$, thus $p=1$, using that $K_{1, 1}(C, L)=\mbox{Ker }\nu_2(L)$, as well as the quadric equation of $G(2, 4)\subset \PP^5$, we observe that $[\zeta(E)]=Q_E$, that is, the classical construction (\ref{rank6quadric}) can be recovered in this Koszul-theoretic setting.
\end{remark}

\begin{remark}
The construction of $[\zeta(E)]$ appears to be, strangely enough, insensitive to the stability of $E$. For instance if $E=A_1\oplus A_2$, where $A_1, A_2$ are base point free line bundles on $C$ contributing to the Clifford index, if we set $r_i:=h^0(C, A_i)-1\geq 1$, then $$0\neq [\zeta(A_1\oplus A_2)]\in K_{r_1+r_2-1}(C, A_1\otimes A_2)$$ is the \emph{Green-Lazarsfeld syzygy} \cite{GL1}. It is the content of Green's Conjecture that in the case of the canonical bundle $K_C$, in some sense, all non-trivial syzygies appear in such a way. We refer to  \cite{V2}, \cite{V3} for a solution of Green's Conjecture for general curves  and to \cite{AF} for a survey.  On the other hand, Mercat's Conjecture can be rewritten as
$$h^0(E)\leq \mbox{sup}\{h^0(A_1)+h^0(A_2): A_1\otimes A_2=\mbox{det}(E),\  \ \  h^i(A_1), h^i(A_2)\geq 2 \ \ \mathrm{ for } \ i=0, 1\}.$$
We conclude that the assignment
$$\mathcal{BN}_C(d, p+3)\ni E\mapsto [\zeta(E)]\in K_{p, 1}(C, \mbox{det}(E))$$ is not expected to produce non-trivial syzygies other than in the range where  Green-Lazarsfeld syzygies are already known to appear.
\end{remark}

This last observation, prompts us to formulate a \emph{Minimal Resolution Conjecture} for the syzygies of curves embedded in projective space by complete linear series. We fix a curve $[C]\in \cM_g$, a complete base point free linear series $L\in W^r_d(C)$, and an integer $1\leq p\leq d-g+1$. Let $\phi_L: C\rightarrow \PP^r$ be the induced morphism.
Using (\ref{koszul}), the condition $K_{p, 1}(C, L)=0$ is equivalent to the injectivity of the restriction map, cf. \cite{PR} or \cite{F5} Proposition 2.3,
\begin{equation}\label{minresconj}
u(C, L): H^0\bigl(\PP^r, \bigwedge^{p-1} M_{\PP^r}(2)\bigr)\stackrel{| _C}\longrightarrow H^0\bigl(C, \bigwedge^{p-1} M_L\otimes L^{\otimes 2}\bigr).
\end{equation}
Note that $M_{\PP^r}=\Omega_{\PP^r}(1)$ and by definition $M_L=\phi_L^*M_{\PP^r}$. The dimensions of both vector spaces appearing in the map (\ref{minresconj}) are independent of $C$ and $L$:
$$h^0(\PP^r, \wedge^{p-1} M_{\PP^r}(2))={r \choose p-1} \frac{(r+1)(r+2)}{p+1}$$
and
$$h^0(C, \wedge^{p-1} M_L\otimes L^{\otimes 2})={r\choose p-1}\Bigl(-\frac{d}{r}(p-1)+2d+1-g\Bigr),$$
where for the second calculation we have used a filtration argument due to Lazarsfeld to show that $H^1(C, \wedge^{p-1} M_L\otimes L^{\otimes 2})=0$. We refer to \cite{F5} Proposition 2.1 for details.

If $\sigma:\mathcal{G}^r_d\rightarrow \cM_g$ is the space of pairs $[C, L]$, where $[C]\in \cM_g$ and $L\in W^r_d(C)-W^{r+1}_d(C)$ is base point free,  there exist vector bundles $\cA$ and $\cB$ over $\mathcal{G}^r_d$ such that,
$$\cA[C, L]=H^0\bigl(\PP^r, \bigwedge^{p-1} M_{\PP^r}(2)\bigr) \ \ \mbox{ and } \ \  \cB[C, L]=H^0\bigl(C, \bigwedge^{p-1} M_L\otimes L^{\otimes 2}\bigr),$$
as well as a vector bundle morphism $u:\cA\rightarrow \cB$
which globalizes the maps $u(C, L)$. We raise the following logical possibility, which is a wide-range generalization of both the Maximal Rank Conjecture $(MRC)_{g, d}^r$ and Green's Conjecture for general curves:

\begin{conjecture}\label{minres} (Minimal Resolution Conjecture)

We fix integers $g, r, d, p\geq 1$ such that $g-d+r\geq 0$, and assume that
\begin{equation}\label{cond1}
r-1-\bigl[\frac{g-1}{2}\bigr]\leq p\leq d-g+1
\end{equation}
and
\begin{equation}\label{cond2}
{r\choose p-1}\Bigl(-\frac{d}{r}(p-1)+2d+1-g-\frac{(r+1)(r+2)}{p+1}\Bigr)+1>\rho(g, r, d).
\end{equation}
Then for a general curve $[C]\in \cM_g$, we have that $K_{p, 1}(C, L)=0$, for all $L\in W^r_d(C)$.
\end{conjecture}

The quantity $[(g-1)/2]$ is the Clifford index of the general curve of genus $g$. Condition (\ref{cond1}) ensures (via Mercat's Conjecture), that non-trivial syzygies of the form $[\zeta(E)]\in K_{p, 1}(C, \mbox{det}(E))$ do not appear in the predicted range. Note that certainly, syzygies of Green-Lazarsfeld type do not appear in $K_{p, 1}(C, L)$, for they would correspond to a pencil $A\in W^1_{r-p}(C)$ and a decomposition $L=A\otimes (L\otimes A^{\vee})$ where $r(A)+r(L\otimes A^{\vee})=p$. But $r-p>\mbox{gon}(C)$, thus $W^1_{r-p}(C)=\emptyset$.

\vskip 3pt
Condition (\ref{cond2}) which implies in particular that $\mbox{rank}(\cA)\leq \mbox{rank}(\cB)$,  expresses the belief/hope that the first degeneracy locus of the morphism $u:\cA\rightarrow \cB$ has the expected dimension and maps to a proper subvariety of $\cM_g$.
Conjecture \ref{minres} implies Mercat's Conjecture. Of course, we regard the Minimal Resolution Conjecture as being vastly more difficult than Mercat's Conjecture, but would still like to point out a remarkable compatibility between two predictions which have been formulated independently of each other.
\begin{remark} When $d=2g-2, r=g-1$, hence $W^{g-1}_{2g-2}(C)=\{K_C\}$, Conjecture \ref{minres} specializes to Green's Conjecture for general curves. This has been established by Voisin \cite{V2}, \cite{V3}. The case $p=1$ of the Minimal Resolution Conjecture is simply the statement $(MRC)_{g, d}^r$ formulated in Section 2.
Various other cases have been proved when $\rho(g, r, d)=0$ and $\mbox{rank}(\cA)=\mbox{rank}(\cB)$, that is, when the failure locus $$\mathfrak{Syz}_{g, d}^r:=\{[C]\in \cM_g: K_{p, 1}(C, L)\neq 0 \mbox{ for  a certain }\ L\in W^r_d(C)\}$$
is a divisor. We mention the case $(g, r, d)=(10, 4, 12)$ cf. \cite{FP}, when the locus $\mathfrak{Syz}_{10, 4}^{12}$ is the $K3$ divisor on $\cM_{10}$, as well as the cases
$(g, r, d)=(16, 7, 21), (22, 10, 30)$ see \cite{F5}.
\end{remark}

\begin{remark}
When $p=1$ condition (\ref{cond1}) is superfluous, being a consequence of (\ref{cond2}). For higher values of $p$ it can happen that
(\ref{cond2}) holds but (\ref{cond1}) fails. An instructive example is that of $2$-canonically embedded curves
$$
C\stackrel{|K_C^{\otimes 2}|}\longrightarrow \PP^{3g-4},
$$
when $d=4g-4, r=3g-4$. Assume $g=4a$, where $a\in \mathbb Z$. For $p=9a-5$, one notices by direct calculation that
$\mbox{rank}(\cA)=\mbox{rank}(\cB)$, and one would expect the degeneracy locus of $u:\cA\rightarrow \cB$ to be a divisor. However inequality (\ref{cond1}) is not satisfied since $p\leq h^0(C, K_C^{\otimes 2})-1-\mathrm{Cliff}(C)$, and indeed by \cite{GL1} we have that
$K_{p, 1}(C, K_C^{\otimes 2})\neq 0$, for every curve $[C]\in \cM_g$. Therefore $u:\cA\rightarrow \cB$ is everywhere degenerate.
\end{remark}

\begin{remark} The name \emph{Minimal Resolution Conjecture} already appears in literature and refers to a statement predicting that if $X\subset \PP^r$ is an embedded projective variety, the resolution of general sets of points $\Gamma\subset X$ is "minimal", being determined by the Hilbert function of $X$ and the cardinality $|\Gamma|$. We refer to \cite{FMP} for a formulation of the most general form of the conjecture and to \cite{EPSW} for the most studied case, that of $X=\PP^r$.  In the case when $X=C\stackrel{|L|}\rightarrow \PP^r$ is a smooth curve of genus $g$ embedded by a very ample linear series $L\in W^r_d(C)$, MRC for points as formulated in \cite{FMP} Corollary 1.8 is equivalent to a collection of vanishing statements for every integer $0\leq i\leq r$:
$$H^1(C, \wedge^i M_L\otimes \xi)=0, \mbox{ for a general line bundle } \xi\in \mathrm{Pic}^j(C), \mbox{ where } j=g-1+\lceil\frac{di}{r}\rceil,$$
and
$$H^0(C, \wedge^i M_L\otimes \xi)=0, \mbox{ for a general line bundle } \xi\in \mathrm{Pic}^j(C), \mbox{ where } j=g-1+\lfloor\frac{di}{r}\rfloor.$$
We do not see an obvious connection between Conjecture \ref{minres} which predicts the minimality of the resolution of $C$ itself, and MRC for \emph{ general points} on $C$. This discrepancy is vividly illustrated when $L=K_C$: Conjecture \ref{minres} specializes to Green's Conjecture for general curves, whereas the Minimal Resolution Conjecture for points boils down to the following equality of cycles in the Jacobian, see \cite{FMP} Theorem 3.1:
$$\Theta_{\bigwedge^i M_{K_C}^{\vee}}=C_{g-i-1}-C_i\subset \mathrm{Pic}^{g-2i-1}(C).$$
This is a statement of a different flavour, for instance it is insensitive to $\mathrm{Cliff}(C)$.
\end{remark}

We record various applications of the Conjecture \ref{minres}:
\begin{proposition}\label{genericdet}
We fix integers $1\leq r\leq g-2$, a general curve $[C]\in \cM_g$ and a general line bundle $L\in \mathrm{Pic}^{g+r}(C)$. Assuming the Minimal Resolution Conjecture for $C$, for any vector bundle $E\in \mathcal{SU}_C(2, L)$,
the following inequality holds:
$$h^0(C, E)<3+\frac{r^2-g}{r+g}.$$
\end{proposition}
\begin{proof} We assume that $E$ is a semistable vector bundle on $C$ with $\mbox{det}(E)=L$ and write $$h^0(C, E)=p+3\geq 3+\frac{r^2-g}{r+g}.$$ First we note that $E$ carries no sub-pencils. Indeed, a general $L\in \mathrm{Pic}^{g+r}(C)$ cannot be expressed as a sum $L=A\otimes A'$, where $h^0(C, A)+h^0(C, A')\geq p+3$. It follows that $0\neq [\zeta(E)]\in K_{p, 1}(C, L)$. The numerical assumption on $p$ is equivalent to the condition
$\mbox{rank}(\cA)\leq \mbox{rank}(\cB)$, in particular Conjecture \ref{minres} implies that $K_{p, 1}(C, L)=0$, which is a contradiction.
\end{proof}

\begin{remark} To derive Proposition \ref{genericdet} we have used a much weakened version of Conjecture \ref{minres}. Precisely, for a general $[C]\in \cM_g$ and $p\geq (r^2-g)/(r+g)$, it suffices to produce \emph{a single example} of a non-special line bundle $L\in \mbox{Pic}^{g+r}(C)$
such that $K_{p, 1}(C, L)=0$, for Theorem \ref{genericdet} to hold true.
\end{remark}
\begin{example} The assumptions of Theorem \ref{genericdet} can be fulfilled for bounded genus. A nice illustration is the case $g=8, r=6$. The Minimal Resolution Conjecture predicts that $K_{2, 1}(C, L)=0$ for a general line bundle $L\in \mathrm{Pic}^{14}(C)$. Equivalently, the ideal of the curve
$C\stackrel{|L|}\longrightarrow \PP^6$ is cut out by quadrics. This has been verified by Verra \cite{Ve} Theorem 5.16, in the course of his proof of the unirationality of $\mm_{14}$. Then from Proposition \ref{genericdet} we deduce that $h^0(C, E)\leq 4$, for any $E\in \mathcal{SU}_C(2, L)$. If we drop the genericity assumption on the determinant bundle $L$, we can find vector bundles having more sections. For instance, there exists a unique vector bundle $E\in \mathcal{SU}_C(2, K_C)$ with $h^0(C, E)=6$, see \cite{M2} Theorem A.
\end{example}

An important particular case of Theorem \ref{genericdet} is when $r=g-2$. In this situation, the predicted vanishing for Koszul cohomology is equivalent to the \emph{Prym-Green Conjecture}, already formulated in \cite{AF} 1.4: If $L\in \mathrm{Pic}^{2g-2}(C)$ is a general line bundle,
\begin{equation}\label{prymgreen}
K_{p, 1}(C, L)=0 \Leftrightarrow p\geq \frac{g-4}{2}.
\end{equation}
The Prym-Green Conjecture predicts in particular, that for $g=2i+6$, the general \emph{paracanonical curve} $C\subset \PP^{g-2}$ embedded by a $\mathfrak g_{2g-2}^{g-2}$, enjoys property $(N_i)$. This statement has important applications to the birational geometry of the moduli spaces
$\mathcal{R}_{g, l}$ parametrizing pairs $[C, \xi]$ where $[C]\in \cM_g$ and $\xi^{\otimes l}=\OO_C$. The Prym-Green Conjecture has been verified for all $g\leq 16$ and details will appear in \cite{EFS}.
\begin{proposition}\label{prgrbn}
For a general curve $[C]\in \cM_g$ with $g\leq 16$, and a general line bundle $L\in \mathrm{Pic}^{2g-2}(C)$, one has the following inequality for all
$E\in \mathcal{SU}_C(2, L)$:
$$h^0(C, E)\leq \frac{g+1}{2}.$$
\end{proposition}
It is worth pointing out that when $L=K_C$, the conclusion of Theorem \ref{prgrbn} no longer holds. If $[C]\in \cM_{2a}$ lies on a $K3$ surface, Mukai and Voisin \cite{V1} have showed that there exists a (unique!) vector bundle $E\in \mathcal{SU}_C(2, K_C)$ with $h^0(C, E)=a+2$. On the other hand, the Brill-Noether subvarieties of $\mathcal{SU}_C(2, K_C)$ have a Lagrangian structure and are governed by different numerical invariants
\cite{BF}, \cite{T3}.
\vskip 3pt

We close, by pointing out that each time a form of the Minimal Resolution Conjecture is known, one can derive a corresponding non-existence result for rank $2$ vector bundles. The following result, is just one example of a statement of this type:
\begin{proposition}\label{genus16}
We fix a general curve $[C]\in \cM_{16}$ and $L\in W^7_{21}(C)$ one of the finitely many linear series residual to a minimal pencil. Then there exist
no semistable bundles $E\in \mathcal{SU}_C(2, L)$ with $h^0(C, E)=5$.
\end{proposition}
\begin{proof} We observe that $\mathrm{Cliff}(C)=\mathrm{Cliff}(L)=7$. Let $E$ be a semistable bundle with $\mbox{det}(E)=L$ and $h^0(C, E)\geq 5$. First we claim that $E$ cannot have sub-pencils. Indeed, if
$$0\rightarrow A\rightarrow E\rightarrow A'\rightarrow 0$$
is an extension with $h^0(C, A)\geq 2$, then $\mbox{deg}(A)\geq 9=\mathrm{gon}(C)$, hence $\mbox{deg}(A')\leq 12$ and $h^0(C, A')\leq 2$ by Brill-Noether theory.
In particular $h^0(C, E)\leq h^0(C, A)+h^0(C, A')\leq 4$, a contradiction.

Thus the bundle $E$ is free of sub-pencils, and then $0\neq [\zeta(E)]\in K_{2, 1}(C, L)$. This implies that $K_{1, 2}(C, L)\neq 0$ as well, in particular
using \cite{F5} Theorem 1.1, the point $[C]\in \cM_{16}$ belongs to the Koszul divisor $\mathfrak{Syz}_{16, 21}^7$, which contradicts the generality assumption on $C$.
\end{proof}


\begin{thebibliography}{EMS}

\bibitem[AF]{AF} M. Aprodu and G. Farkas, {\em{Koszul cohomology and applications to moduli}}, arXiv:0811.3117, to appear in "Grassmannians,
vector bundles and moduli spaces", Proceedings Clay Math. Institute, Volume 14 (2011).
\bibitem[AN]{AN} M. Aprodu and J. Nagel, {\em{Non-vanishing for Koszul cohomology of curves}}, Comment. Math. Helvetici \textbf{82} (2007), 617-628.
\bibitem[ACGH]{ACGH}
E. Arbarello, M. Cornalba, P. Griffiths and J. Harris, {\em{Geometry
of algebraic curves}}, Grundlehren der mathematischen Wissenschaften
\textbf{267}, Springer Verlag 1985.
\bibitem[vB]{vB} H. -Ch. Graf von Bothmer, {\em{Generic syzygy schemes}}, J. Pure Applied Algebra, \textbf{208} (2007), 867-876.
\bibitem[BF]{BF} A. Bertram and B. Feinberg, {\em{On stable rank two bundles with canonical determinant}}, in: "Algebraic Geometry" ed. P. Newstead,  Marcel Dekker 1998, 259-269.
\bibitem[BV]{BV} S. Brivio and A. Verra, {\em{The theta divisor of $SU_C(2, 2d)^s$ is very ample if $C$ is not hyperelliptic}}, Duke Math. Journal
\textbf{82} (1996), 503-552.
\bibitem[EFS]{EFS} D. Eisenbud, G. Farkas and F.-O. Schreyer, {\em{The Prym-Green Conjecture}}, in preparation.
\bibitem[EPSW]{EPSW} D. Eisenbud, S. Popescu, F.-O. Schreyer and Ch. Walter, {\em{Exterior algebra methods for the Minimal Resolution Conjecture}}, Duke Math. Journal \textbf{112} (2002), 379-395.

\bibitem[F1]{F1} G. Farkas, {\em{Brill-Noether loci and the gonality stratification of $\cM_g$}}, J. reine angew. Mathematik \textbf{539} (2001), 185-200.
\bibitem[F2]{F2} G. Farkas, {\em{Gaussian maps, Gieseker-Petri loci and large theta-characteristics}}, J. reine angew. Mathematik \textbf{581} (2005), 151-173.
\bibitem[F3]{F3} G. Farkas, {\em{Koszul divisors on moduli spaces of
curves}}, American Journal of Math. \textbf{131} (2009),
819-869.
\bibitem[F4]{F4} G. Farkas, {\em{Rational maps between moduli spaces of curves and Gieseker-Petri divisors}}, J. Algebraic Geometry \textbf{19} (2010), 243-284.
\bibitem[F5]{F5} G. Farkas, {\em{Syzygies of curves and the effective cone of $\mm_g$}}, Duke Math. Journal \textbf{135} (2006), 53-98.
\bibitem[FMP]{FMP} G. Farkas, M. Musta\c{t}\u{a} and M. Popa, {\em{Divisors on $\cM_{g, g+1}$ and the Minimal Resolution Conjecture for points on canonical curves}}, Ann. Scient. \'{E}cole Normale Sup. \textbf{36} (2003), 553-581.
\bibitem[FP]{FP} G. Farkas and M. Popa, {\em{Effective divisors on
$\mm_g$, curves on $K3$ surfaces and the Slope Conjecture}}, Journal
of Algebraic Geometry \textbf{14} (2005), 151-174.
\bibitem[GL1]{GL1} M. Green and R. Lazarsfeld, {\em{The non-vanishing of certain Koszul cohomology groups}}, J. Differential Geometry \textbf{19}
(1984), 168-170.
\bibitem[GL2]{GL} M. Green and R. Lazarsfeld, {\em{Special divisors on curves on a $K3$ surface}}, Inventiones Math. \textbf{89} (1987), 357-370.
\bibitem[GMN]{GMN} I. Grzegorczyk, V. Mercat and P. E. Newstead, {\em{Stable bundles of rank $2$ with $4$ sections}}, arXiv:1006.1258.
\bibitem[H]{H} J. Harris, {\em{Curves in projective space. With the collaboration of D. Eisenbud}}, University of Montreal 1982.
\bibitem[HM]{HM} J. Harris and D. Mumford, {\em{On the Kodaira
dimension of $\mm_g$}}, Inventiones Math. \textbf{67} (1982), 23-88.
\bibitem[K]{K} A. Knutsen, {\em{Smooth curves on projective $K3$ surfaces}}, Math. Scandinavica \textbf{90} (2002), 215-231.
\bibitem[L]{L} Y. Laszlo, {\em{Un Th\'eor\`eme de Riemann pour les diviseurs theta sur les espaces de modules de fibres stables sur une courbe}}, Duke Math. Journal \textbf{64} (1991), 333-347.
\bibitem[LN]{LN} H. Lange and P. E. Newstead, {\em{Clifford indices for vector bundles on curves}}, arXiv:0811.4680, in: Affine Flag Manifolds and Principal Bundles (A. H. W. Schmitt editor), Trends in Mathematics, 165-202, Birkh\"auser 2010.
\bibitem[LMN]{LMN} H. Lange, V. Mercat and P. E. Newstead, {\em{On an example of Mukai}}, arXiv:1003.4007.
\bibitem[Me]{Me} V. Mercat, {\em{Clifford's theorem and higher rank vector bundles}}, International Journal of Math. \textbf{13} (2002), 785-796.
\bibitem[M1]{M1} S. Mukai, {\em{Curves and $K3$ surfaces of genus eleven}}, in: Moduli of vector bundles, Lecture Notes in Pure and Appl. Math. 179 Dekker (1996), 189-197.
\bibitem[M2]{M2} S. Mukai, {\em{Curves and Grassmannians}}, in: Algebraic Geometry and Related Topics, Inchon 1992,  (J.-H. Yang, Y. Namikawa, K. Ueno editors), 19-40, International Press 1993.
\bibitem[PR]{PR} K. Paranjape and S. Ramanan, {\em{On the canonical ring of a curve} }, in: Algebraic Geometry and commutative Algebra (1987), 503-516.
\bibitem[T1]{T1} M. Teixidor i Bigas, {\em{On the Gieseker-Petri map for rank two vector bundles}}, Manuscripta Math., \textbf{75} (1992), 375-382.
\bibitem[T2]{T2} M. Teixidor i Bigas, {\em{Existence of coherent systems of rank two and dimension four}}, Collectanea Math. \textbf{58} (2007), 193-198.
\bibitem[T3]{T3} M. Teixidor i Bigas, {\em{Petri map for rank two bundles with canonical determinant}}, Compositio Math. \textbf{144} (2008), 705-720.
\bibitem[Ve]{Ve} A. Verra, {\em{The unirationality of the moduli
space of curves of genus $14$ or lower}}, Compositio Math. \textbf{141} (2005),
1425-1444.
\bibitem[V1]{V1} C. Voisin, {\em{Sur l'application de Wahl des courbes satisfaisant la condition de Brill-Noether-Petri}}, Acta Math. \textbf{168} (1992), 249-272.
\bibitem[V2]{V2} C. Voisin,
{\em{Green's generic syzygy conjecture for curves of even genus lying on a $K3$
surface}} J. European Math. Society, \textbf{4}, (2002), 364-404.
\bibitem[V3]{V3}
C. Voisin, {\em{Green's canonical syzygy conjecture for generic curves of odd genus}},
Compositio Math. \textbf{141} (2005), 1163-1190.
\end{thebibliography}
\end{document}